\documentclass[12pt]{article}%
\usepackage{amsmath}
\usepackage{amsfonts}
\usepackage{amssymb}
\usepackage{graphicx}%
\providecommand{\U}[1]{\protect\rule{.1in}{.1in}}
\newtheorem{theorem}{Theorem}

\newtheorem{example}[theorem]{Example}

\newtheorem{lemma}[theorem]{Lemma}

\newtheorem{proposition}[theorem]{Proposition}
\newtheorem{remark}[theorem]{Remark}

\newenvironment{proof}[1][Proof]{\noindent\textbf{#1.} }{\ \rule{0.5em}{0.5em}}
\begin{document}

\title{Burkholder Theorem in Riesz spaces }
\author{Youssef Azouzi\\{\small Research Laboratory of Algebra, Topology, Arithmetic, and Order}\\{\small Faculty of Mathematical, Physical and Natural Sciences of Tunis}\\{\small Tunis-El Manar University, 2092-El Manar, Tunisia}\\Kawtar Ramdane\\Moroccan Polytechnic Research and Innovation Center (CMRPI)}
\date{}
\maketitle

\begin{abstract}
The main purpose of this paper is to give a vector lattice version of a
Theorem by Burkholder about convergence of martingales. The proof is based on
a vector lattice analogue of Austin's sample function theorem, proved recently
by Grobler, Labuschagne and Marraffa and on a new characterization of elements
of the sup-completion of a universally complete vector lattice which do not
belong to the space.

\end{abstract}

\section{Introduction}

This paper aims to generalize the following theorem due to Burkholder in the
setting of vector lattices. Our proof presented here follows Chen \cite{L-25}.

\begin{theorem}
\label{TB}Let $(\Omega,\mathcal{F},P)$ be a probability space. Consider two
martingales $f=\left(  f_{n}\right)  _{n\geq1}$ and $g=\left(  g_{n}\right)
_{n\geq1}$ with respect to the same filtration $(\mathcal{F}_{n})_{n\geq1}$
such that $||f||_{1}:=\sup\left\Vert f_{n}\right\Vert _{1}<\infty$
\textit{and} $S_{n}(g)\leq S_{n}(f)$ \textit{for} $n=1,2,\ \ldots$.
\textit{Then} $g$ \textit{converges} almost surely.
\end{theorem}

Here $S_{n}\left(  h\right)  $ denotes the Quadratic Variation Process
associated to $h,$ that is, $S_{n}\left(  h\right)  =%
{\textstyle\sum\limits_{k=1}^{n}}
\left(  h_{k}-h_{k-1}\right)  ^{2}.$

In the framework of vector lattices we will exclusively consider Dedekind
complete vector lattices, so every bounded nonempty subset of $E$ has a
supremum and infimum. Comparing with the classical probability theory, one can
think of the space $L_{1}=L_{1}\left(  \Omega,\mathcal{F},\mathbb{P}\right)  $
consisting of all integrable real random variables. The space $E^{u},$ the
universal completion of $E,$ plays the same role as $L_{0}$ does in the
classical theory, while $E_{s},$ the sup-completion of $E$ corresponds to the
cone of random variables taking their values in $\mathbb{R}\cup\left\{
\infty\right\}  $ which are bounded below by some integrable variable. Recall
that $E^{u}$ is Dedekind complete and laterally complete, where the later
means that every subset of $E^{u}$ which consists of mutually disjoint
positive elements, has a supremum in $E^{u}.$ However, an arbitrary subset of
$E^{u}$ need not have a supremum. This lack of supremum could be made up by
considering the sup-completion $E_{s}$ of $E.$ This concept is due to Donner
in \cite{b-1665} and we refer to \cite{b-1665,L-444} for the definition and
properties of the sup-completion $E_{s}$. We just recall here that $E_{s}$ is
a unique Dedekind complete ordered cone that contains $E$ as a sub-cone of its
group of invertible elements and in which every nonempty subset has a
supremum. The role of a probability measure is played by a conditional
expectation operator, i.e., a strictly positive order continuous projection
$T$ mapping the vector lattice into itself onto a Dedekind complete Riesz
subspace with $Te=e$ form some weak order unit $e.$ A great part of this
theory has been developed by Kuo, Labuschagne and Watson (see for example
\cite{L-24,L-32,L-03}) for the discrete case and essentially by Grobler,
Labuschagne for the continuous case (see for example
\cite{L-10,L-06,L-21,L-12,L-13}). The central result of this work is to get a
generalization of Theorem \ref{TB}. Several nice results can be discovered in
the way to reach our purpose. We believe that it is the case here for Theorem
\ref{L4} and Theorem \ref{ThK}, which concern order structure and are totally
independent of the probability theory. Both of them are crucial steps to prove
our main result. The first result provides a characterization of elements in
$E_{s}$ that do not belong to $E,$ where $E$ is a universally complete vector
lattice. It states that a positive element $x$ in $E_{s}$ belongs to $E$ if
and only if $\inf\limits_{\lambda\in\left(  0,\infty\right)  }P_{\left(
x-\lambda e\right)  ^{+}}e=0,$ where $P_{u}$ denotes the order projection on
the band generated by $u.$ This condition is the abstract translation of the
fact that $\mathbb{P}\left(  \left\vert X\right\vert \geq\lambda\right)
\downarrow0$ as $\lambda\longrightarrow\infty,$ which means that the random
variable $X$ is finite. For more informations about the connection between
$E^{u}$ and $E_{s}$ we refer to \cite{L-444}. The second result, Theorem
\ref{ThK}, provides a sufficient condition for a sequence $\left(
x_{n}\right)  _{n\geq1}$ in $E$ to converge in order in the space $E^{u}.$
This condition says that it is enough to prove that for some sequence of band
projections $\left(  P_{k}\right)  $ satisfying $P_{k}\uparrow I,$ the
sequences $\left(  P_{k}x_{n}\right)  _{n\geq1},k=1,2,...$ converges in order.
In the proof of this result we use the notion of unbounded order convergence,
which characterizes the notion of almost everywhere convergence in the
classical theory, and has received increasing attention in recent years (see
for example \cite{L-65,L-719}). In order to prove some other intermediate
results we are led to use the notion of Riemann integral in vector lattices
introduced by the authors in \cite{L-311}. All these results are proved in
Section 2. In that section we collect all results we need to prove the main
theorem of this paper (Theorem \ref{Main}). Its proof is presented in Section 3.

As we have pointed out above, we will consider a Dedekind complete vector
lattice $E$ and we assume that $E$ has a weak order unit $e.$ This space will
be equipped with a conditional expectation operator $T$ such that $Te=e.$ This
distinguished unit $e$ plays the same role as the constant random variable
$\mathbf{1}$ does in the classical theory.$.$It was shown in \cite{L-24} that
there exists a largest Riesz subspace of $E^{u},$ called the \textit{natural}
\textit{domain }of $T$, and denoted by $L^{1}\left(  T\right)  ,$ to which $T$
extends uniquely to a conditional expectation which will denoted again by $T$.
It is worth noting that $L^{1}\left(  T\right)  $ is a Dedekind complete ideal
of $E^{u}.$ It is also $T$-universally complete, that is, for every increasing
net $\left(  x_{\alpha}\right)  $ in $L^{1}\left(  T\right)  $ such that
$\left(  Tx_{\alpha}\right)  _{\alpha}$ is bounded in $E^{u},$ the supremum
$\sup x_{\alpha}$ exists in $L^{1}\left(  T\right)  .$ In \cite{L-06} Grobler
introduced the notion of Daniell Integral in vector lattices and used it to
develop a kind of functional calculus. These tools has been considered to
study the spaces $\mathcal{L}^{p}(T)$ for $p\in\lbrack1,\infty)$ in
\cite{L-180} by the first author and Trabelsi. For $p\in\lbrack1,\infty)$ the
space $L^{p}\left(  T\right)  $ is equipped with the a vector valued norm
$\left\Vert {}\right\Vert _{p,T}$ defined by%
\[
\left\Vert f\right\Vert _{p,T}=\left(  T\left\vert f\right\vert ^{p}\right)
^{1/p},\qquad f\in L^{p}\left(  T\right)  .
\]
Notions of filtrations, martingales have been defined in the measure-free
framework of vector lattices by means of operators. All of them are needed to
give an abstract formulation of Burkholder Theorem. Let us recall some
definitions. A filtration in $E$ is a sequence of conditional expectation
operators $\left\{  T_{n}:n\geq1\right\}  $ with $T_{1}=T$ and $T_{i}%
T_{j}=T_{j}T_{i}=T_{\min\left(  i,j\right)  }$ for all $i,j\geq1.$%
(\cite[Definition 3.1]{L-32}). A martingale is a sequence $\left(
x_{n}\right)  _{n\geq1}$ of elements of $E$ such that $T_{n}x_{n+1}=x_{n}$ for
all $n\geq1.$ Consider now two martingales\textit{ }$(f_{n})_{n\geq1},$ and
$(g_{n})_{n\geq1}$ and define $\left\Vert f\right\Vert _{1}=\sup
\limits_{n}T\left\vert f_{n}\right\vert \in E_{s}$. Keeping the same notations
as in the concrete case, it turns out that $g$ converges in order in $E^{u}$
whenever $S_{n}\left(  g\right)  \leq S_{n}\left(  f\right)  $ for $n=1,2,...$
and $\left\Vert f\right\Vert _{1}$ exists in $E^{u}$. For the proof of this
result we need also the notions of stopping time and stopped process in the
setting of Riesz spaces. We recall from \cite{L-32} that a stopping time is an
increasing sequence $(P_{i})_{i\geq1}$ of band projections on $E$ such that

\begin{center}
$P_{i}T_{j}=T_{j}P_{i}$ whenever $1\leq i\leq j$.
\end{center}

Given a filtration $\left(  T_{n}\right)  ,$ an adapted process $f=\left(
f_{n}\right)  ,$ that is, $\left(  f_{n}\right)  $ is a sequence in $E$ with
$f_{n}\in R(T_{n})$, for all $n\in\mathbb{N}$, and a stopping time $P=\left(
P_{n}\right)  ,$ we can define the stopped process $(f_{p},T_{p})$ by putting
\[
f_{P}=\sum\limits_{i=1}^{\infty}(P_{i}-P_{i-1})f_{i}.
\]
One of the key ideas in the proof of Burkholder Theorem is a vector lattice
version of Austin Theorem, which is the main result of \cite{L-21}. For more
informations about Riesz spaces we refer to \cite{b-240,Z1997} and for
probability theory we refer to \cite{b-23} and \cite{CT1988}.

\section{Several tools}

This section contains all results needed to prove our main Theorem (Theorem
\ref{Main}). We start be giving a brief review of unbounded order convergence.
This kind of convergence is a satisfactory abstraction of almost everywhere
convergence A net $\left(  x_{\alpha}\right)  _{\alpha\in A}$ in a vector
lattice $E$ is said to be \textit{order convergent} to $x,$ and we write
$x_{\alpha}\overset{o}{\longrightarrow}x$, if there exists a net $\left(
y_{\beta}\right)  _{\beta\in B}$ such that $y_{\beta}\downarrow0$ and for each
$\beta\in B$ there exists $\alpha_{\beta}\in A$ satisfying $\left\vert
x_{\alpha}-x\right\vert \leq y_{\beta}$ for all $\alpha\geq\alpha_{\beta}.$ A
net $\left(  x_{\alpha}\right)  $ is said to \textit{converge in unbounded
order} to $x,$ and we write $x_{\alpha}\overset{uo}{\longrightarrow}x,$ if
$\left\vert x_{\alpha}-x\right\vert \wedge y\overset{o}{\longrightarrow}0$ for
every $y\in E_{+}.$ If $E$ has a weak unit $e$ this can be reduced to
$\left\vert x_{\alpha}-x\right\vert \wedge e\overset{o}{\longrightarrow}%
0.$(see Theorem 2.2. in \cite{L-63}) It should be noted that these
convergences agree for sequences if the space is universally complete
\cite[Corollary 3.12]{L-65}. For more informations about $uo$-convergence the
reader is referred to papers \cite{L-63} and \cite{L-65}. It is sometimes
convenient to use series instead of sequences. The notation $%
{\textstyle\sum\limits_{k=1}^{\infty}}
x_{k}$ means the order limit of the sequence $S_{n}=%
{\textstyle\sum\limits_{k=1}^{n}}
x_{k}.$ Our first important result in this section provides a sufficient
condition for a sequence $\left(  x_{n}\right)  $ in $E$ to be order
convergent in $E^{u}.$ Let us record a lemma which will be of great use for
us; see e.g., \cite[Corollary 32.4]{Z1997}.

\begin{lemma}
\label{L6}Let $E$ be a vector lattice. If the projection bands $B_{1}%
,B_{2},B_{3},B_{4}$ in $E$ with corresponding band projections $P_{1}%
,P_{2},P_{3},P_{4}$ satisfy $B_{4}\subset B_{3}\subset B_{2}\subset B_{1}$
then $\left(  P_{1}-P_{2}\right)  \perp\left(  P_{3}-P_{4}\right)  $ and so
$\left(  P_{1}-P_{2}\right)  u\wedge\left(  P_{3}-P_{4}\right)  u=0$ for every
$u\in E^{+}.$
\end{lemma}

\begin{theorem}
\label{ThK}Let $E$ be a Dedekind complete vector lattice and consider a
sequence of band projections $\left(  P_{k}\right)  _{k\geq1}$ such that
$P_{k}\uparrow I$ and a sequence $\left(  x_{n}\right)  _{n\geq1}$ of elements
of $E.$ If for each $k$ the sequence $\left(  P_{k}x_{n}\right)  _{n\geq1}$
converges in order in $E$, then $\left(  x_{n}\right)  _{n\geq1}$ converges in
order in $E^{u}.$
\end{theorem}

\begin{proof}
By considering the positive part and the negative part of $x_{n}$ we may
assume, without loss of generalities, that $x_{n}\in E^{+}$ for all $n.$ We
denote by $y_{k}$ the order limit of $\left(  P_{k}(x_{n})\right)  _{n\geq1}$.
The sequence $\left(  y_{k}\right)  _{k\geq1}$ is increasing because $\left(
P_{k}\right)  _{k\geq1}$ is so. Moreover, we have%
\[
P_{k-1}y_{k}=\lim\limits_{n}P_{k-1}P_{k}x_{n}=y_{k-1}.
\]
Since we have also $y_{k}=P_{k}y_{k}$ we obtain%
\[
y_{k}=%
{\textstyle\sum\limits_{i=1}^{k}}
\left(  P_{i}-P_{i-1}\right)  y_{i},
\]
\newline where $P_{0}=0.$ Now since $\left(  P_{i}-P_{i-1}\right)  $ is a
sequence of disjoint projections we see that%
\[
y=\sup y_{k}=%
{\textstyle\sum\limits_{i=1}^{\infty}}
\left(  P_{i}-P_{i-1}\right)  y_{i}%
\]
exists in $E^{u}.$ We claim now that $x_{n}\overset{o}{\longrightarrow}y$ in
$E^{u}$. As it was noted above, it is sufficient to prove that $\left(
x_{n}\right)  _{n\geq1}$ converges to $y$ in unbounded order. To this end let
$z\in E_{+}$ and observe that%
\begin{align*}
\left\vert x_{n}-y\right\vert \wedge z  &  =P_{k}(\left\vert x_{n}%
-y\right\vert \wedge z)+P_{k}^{d}(\left\vert x_{n}-x\right\vert \wedge z)\\
&  \leq P_{k}(\left\vert x_{n}-y\right\vert \wedge z)+P_{k}^{d}z,
\end{align*}
which yields%
\[
\limsup\limits_{n\longrightarrow\infty}\left\vert x_{n}-y\right\vert \wedge
z\leq P_{k}\left\vert y_{x}-y\right\vert +P_{k}^{d}z\leq\left\vert
y_{k}-y\right\vert +P_{k}^{d}z.
\]
\newline and the result follows since $P_{k}^{d}z\overset{o}{\longrightarrow
}0$ and $\left\vert y-y_{k}\right\vert \overset{o}{\longrightarrow}0$ as
$k\longrightarrow\infty.$
\end{proof}

\begin{remark}
\label{R2}Does the above theorem remain valid for nets? As order convergence
and unbounded order convergence do not agree for nets even in universally
complete vector lattices (see \cite[Example 6]{L-174}), it seems more
reasonable to check for a counterexample. However, a slight modification of
the above proof yields that if we take a net $\left(  x_{a}\right)  $ instead
of a sequence in Theorem \ref{ThK}, we can conclude that $\left(  x_{\alpha
}\right)  $ is $uo$-convergent in $E^{u}$ and its $uo$-limit $y$ satisfies
$y=\lim y_{k}=\sup y_{k},$ where $y_{k}$ is the order limit of $\left(
P_{k}x_{\alpha}\right)  _{\alpha\in A}.$ Now if we look for a more general
result we will take a net of band projections. In this case the construction
of the $uo$-limit of $\left(  x_{\alpha}\right)  $ is less obvious, but the
result remains true as will be shown in the next result.
\end{remark}

\begin{theorem}
\label{T2}Let $E$ be a Dedekind complete vector lattice and consider a net of
band projections $\left(  P_{\gamma}\right)  _{\gamma\in\Gamma}$ such that
$P_{\gamma}\uparrow I$ and a net $\left(  x_{\alpha}\right)  _{\alpha\in A}$
of elements in $E.$ If for each $\gamma$ the net $\left(  P_{\gamma}x_{\alpha
}\right)  _{\alpha\in A}$ is order convergent, then $\left(  x_{\alpha
}\right)  $ is $uo$-convergent in $E^{u}.$
\end{theorem}

\begin{proof}
We denote $y_{\alpha}=\lim\limits_{\gamma}P_{\gamma}x_{\alpha}$ for each
$\alpha\in A.$ As in the proof of Theorem \ref{ThK} we can show that $\left(
y_{\alpha}\right)  _{\alpha\in A}$ is an increasing net, but it is not obvious
how to show that its supremum is still in $E^{u}$. We will prove instead that
the net $\left(  x_{\alpha}\right)  $ is $uo$-Cauchy and conclude then by
\cite[Theorem 17]{L-444}. To this end let $y\in E_{+}$ and observe that
\begin{align*}
\left\vert x_{\alpha}-x_{\beta}\right\vert \wedge y  &  =P_{\gamma}\left\vert
x_{\alpha}-x_{\beta}\right\vert \wedge y+P_{\gamma}^{d}\left\vert x_{\alpha
}-x_{\beta}\right\vert \wedge y\\
&  \leq P_{\gamma}\left(  \left\vert x_{\alpha}-y_{\gamma}\right\vert
+P_{\gamma}\left\vert x_{\beta}-y_{\gamma}\right\vert \right)  +P_{\gamma}%
^{d}y.
\end{align*}
It follows that $\limsup\limits_{\left(  \alpha,\beta\right)  }\left\vert
x_{\alpha}-x_{\beta}\right\vert \wedge y\leq P_{\gamma}^{d}y$ for all
$\gamma\in\Gamma$ and then $\limsup\limits_{\left(  \alpha,\beta\right)
}\left\vert x_{\alpha}-x_{\beta}\right\vert \wedge y=0$ which proves that
$\left(  x_{\alpha}\right)  $ is $uo$-Cauchy as required.
\end{proof}

As it was mentioned in Remark \ref{R2}, the order convergence can not be
expected in the above Theorem.

\begin{example}
Let $E=\ell_{\infty}$ and consider the net $\left(  x_{\alpha}\right)
_{\alpha\in A}$ defined by: $x_{\left(  p,q\right)  }=q%
{\textstyle\sum\limits_{p\leq k\leq q}}
e_{k},$ where the indexed set, $A=\left\{  \left(  p,q\right)  \in
\mathbb{N}\times\mathbb{N}:p\leq q\right\}  ,$ is ordered cordinatwise and
$\left(  e_{k}\right)  $ is the standard basis of $\ell_{\infty}$. Let $P_{k}$
be the band projection on the principal band generated by $e_{1}+...+e_{k}.$
It is easy to check that $P_{k}x_{\alpha}\overset{o}{\longrightarrow}0$ for
every $k.$ However, $\left(  x_{\alpha}\right)  $ is not order convergent in
$\ell_{\infty}^{u}$ because it hasn't an order bounded tail. Indeed if
$\alpha_{0}=\left(  p_{0},q_{0}\right)  $ is fixed then for $q>q_{0}\geq
p_{0}$ and $\alpha_{q}=\left(  p_{0},q\right)  $ we have $\alpha_{q}\geq
\alpha_{0}$ and $x_{\alpha_{q}}\geq qe_{p_{0}}.$
\end{example}

It is frequent in classical probability theory to consider random variables
with values in the extended real line, and most often in $\mathbb{R}_{\infty
}=\mathbb{R}\cup\left\{  \infty\right\}  ,$ and then it is question to check
if these variables are finite almost everywhere. The corresponding situation
in our model of measure-free theory, random variables finite almost everywhere
correspond to elements of $E^{u},$ and random variables with values in
$\mathbb{R}_{\infty}$ correspond to elements of $E_{s},$ the sup-completion of
$E$ and it was a challenge for us to recognize `finite elements' of $E_{s},$
those that belong to $E^{u}.$ In \cite[Theorem 14]{L-444}, the first author
obtains a characterization of elements of $E_{s}$ which do not belong to
$E^{u}$ for a Dedekind complete vector lattice $E;$ that characterization is a
crucial step in the proof of the main theorem in \cite{L-444}, which states
that a vector lattice is uo-complete if and only if it is universally
complete, so to prove that a net in $E^{u}$ is uo-convergent it is sufficient
to show that it is $uo$-Cauchy, as it has been done in the proof of Theorem
\ref{T2}. As the proof of \cite[Theorem 14]{L-444} is involved and will be
only applied in the case of universally complete spaces, we provide here an
alternative proof of this result in that special case, which is less involved.
Let us first recall some notations and facts. Consider a Dedekind complete
vector lattice $E$ with a weak order unit $e$ and denote by $E_{s}$ its
sup-completion. If $a\in E^{+},$ we denote by $P_{a}$ the band projection on
the principal band generated by $a.$ It satisfies:%
\[
P_{a}(x)=\sup\limits_{k}(x\wedge ka),\qquad\text{for all }x\in E^{+}.
\]
We shall use this formula to extend $P_{a}$ for all positive elements $x$ in
$E_{s}.$ On the other hand if $a$ is a positive element in $E_{s}$, we can
define the map $\varphi_{a}$ on $E$ by putting%
\[
\varphi_{a}\left(  x\right)  =\sup\limits_{n\in\mathbb{N}}\left(  x\wedge
na\right)  =\sup\limits_{\lambda>0}\left(  x\wedge\lambda a\right)
\]
for all $x\in E_{+}$ and then $\varphi_{a}\left(  x\right)  =\varphi
_{a}\left(  x^{+}\right)  -\varphi_{a}\left(  x^{-}\right)  $ for all $x\in
E.$ The next lemma shows that $\varphi_{a}$ is a band projection on $E.$

\begin{lemma}
\cite[Lemma 4]{L-444}Let $E$ be a Dedekind complete vector lattice with a weak
order unit $e.$ If $a\in E_{s}^{+},$ then $\varphi_{a}=P_{\varphi_{a}\left(
e\right)  }.$
\end{lemma}

\begin{remark}
It is easily checked that $P_{a}(x)=\sup\limits_{\beta}P_{a}(x_{\beta})$ for
each net $\left(  x_{\beta}\right)  $ in $E_{+}$ which increases to $x\in
E_{s}^{+}$. Instead of $\varphi_{a}$ we prefer the notation $P_{a}$ even if
$a$ in a positive element in $E_{s}.$
\end{remark}

The following lemma will be needed in the proof of the second main result of
this section.

\begin{lemma}
\label{L7}Let $E$ be a Dedekind complete vector lattice with weak order unit
and let $P$ and $Q$ be two disjoint band projections. Then%
\[
\left(  P+Q\right)  z=\left(  P\vee Q\right)  z\qquad\text{for all }z\in
E_{s}^{+}.
\]

\end{lemma}

\begin{proof}
Let $\left(  z_{\alpha}\right)  _{\alpha\in\Gamma}$ be a net in $E^{+}$ such
that $z_{\alpha}\uparrow z.$ We have seen that $Rz=\sup Rz_{\alpha}$ for all
band projection $R.$ Since $\left(  P\vee Q\right)  z_{\alpha}=Pz_{\alpha
}+Qz_{\alpha}\leq Pz+Qz$ holds for all $\alpha$ we get the inequality%
\[
\left(  P\vee Q\right)  z\leq Pz+Qz.
\]
For the reverse inequality we need the following fact: If $A,B$ are two
non-empty subsets of $E$ then $\sup\left(  A+B\right)  =\sup A+\sup B$ in
$E_{s}$ (see the proof of \cite[Theorem 1.4]{b-1665}). Apply this to subsets
$A=\left\{  Pz_{\alpha}\right\}  $ and $B=\left\{  Qz_{\alpha}\right\}  $ to
obtain%
\[
Pz+Qz=\sup\limits_{\alpha,\beta}\left(  Pz_{\alpha}+Qz_{\beta}\right)  .
\]
Now, given $\alpha,\beta$ in $\Gamma,$ there exists $\gamma\in\Gamma$ such
that $\gamma\geq\alpha$ and $\gamma\geq\beta,$ hence
\[
Pz_{\alpha}+Qz_{\beta}\leq Pz_{\gamma}+Qz_{\gamma}=\left(  P\vee Q\right)
z_{\gamma}\leq\left(  P\vee Q\right)  z.
\]
Taking the supremum over $\alpha,\beta$ we get
\[
\left(  P+Q\right)  z\leq\left(  P\vee Q\right)  z,
\]
which yields the desired result.
\end{proof}

\begin{theorem}
\label{L4}Let $E$ be a universally complete vector lattice with a weak order
unit $e$ and let $x\in E_{s},$ the sup-completion of $E$. Then $x\in E$ if and
only if $\inf\limits_{k>0}P_{(x-ke)^{+}}e=0.$
\end{theorem}

\begin{proof}
The forward direction is obvious. For the converse direction assume that
$\inf\limits_{k}P_{(x-ke)^{+}}e=0.$ We will prove that $x\leq y$ for some
$y\in E,$ which shows that $x\in E$ because $E$ is a tight imbedding cone in
$E_{s}$ (see \cite{b-1665}). Define two sequence of band projections by
putting for $k=1,2,...,$%
\[
P_{k}=P_{(x-ke)^{+}}^{d}=P_{e-P_{\left(  x-ke\right)  ^{+}}e},\qquad
Q_{k}=P_{k}-P_{k-1},
\]
with $P_{0}=0.$ Then $Q_{i}y\wedge Q_{j}y=0$ for all $i,j\in\mathbb{N}$ with
$i\neq j,$ and $y\in E_{+}$ by Lemma \ref{L7}, and $P_{k}\uparrow I$ by our
assumption. It follows therefore that%
\[
x\wedge le=\sup\limits_{n}P_{n}\left(  x\wedge le\right)  =\sup\limits_{n}%
{\textstyle\sum\limits_{k=1}^{n}}
Q_{k}(x\wedge le)=%
{\textstyle\sum\limits_{k=1}^{\infty}}
Q_{k}(x\wedge le).
\]
According to Lemma \ref{L7} we have that%
\[
Q_{m}\left(  z\right)  +Q_{n}(z)=Q_{m}\vee Q_{n}(z),\text{ for all }z\in
E_{s}^{+}.
\]
It follows that%
\begin{align*}
x  &  =\sup\limits_{l}\left(  x\wedge le\right)  =\sup\limits_{l}%
\sup\limits_{n}%
{\textstyle\sum\limits_{k=1}^{n}}
Q_{k}\left(  x\wedge le\right) \\
&  =\sup\limits_{n}\sup\limits_{l}%
{\textstyle\sum\limits_{k=1}^{n}}
Q_{k}\left(  x\wedge le\right)  =\sup\limits_{n}%
{\textstyle\sum\limits_{k=1}^{n}}
Q_{k}x\\
&  =%
{\textstyle\sum\limits_{k=1}^{\infty}}
Q_{k}x.
\end{align*}
\newline Now from the inequality%
\[
x\leq%
{\textstyle\sum\limits_{k=1}^{\infty}}
Q_{k}(x-ke)^{+}+%
{\textstyle\sum\limits_{k=1}^{\infty}}
Q_{k}ke\text{ },
\]
and the fact that $Q_{k}(x-ke)^{+}=0$\ for all $k\in\mathbb{N},$ we obtain
\[
x\leq%
{\textstyle\sum\limits_{k=1}^{\infty}}
kQ_{k}e.
\]
But $%
{\textstyle\sum\limits_{k=1}^{\infty}}
kQ_{k}e$ belongs to $E$ as a supremum of positive disjoint sequence, so $x\in
E$ as required.
\end{proof}

\begin{remark}
It is worth noting that the proof presented above still valid if we assume
only that $E$ is $\sigma$-universally complete.
\end{remark}

The next corollary occurs by combining Theorem \ref{L4} with Doob Maximal
Inequality \cite[Theorem 6.1]{L-06}. Recall that a net $\left(  x_{\alpha
}\right)  $ in $E$ is said to be $T$-bounded if $\left(  T\left\vert
x_{\alpha}\right\vert \right)  $ is bounded in $E.$

\begin{lemma}
\label{L5}Let $E$ be a Dedekind complete vector lattice with a weak order unit
$e$ equipped with a conditional expectation $T$ with $Te=e.$ If $f=\left(
f_{n},T_{n}\right)  $ is a $T$-bounded in $E$ then $f^{\ast}:=\sup\left\vert
f_{n}\right\vert \in E^{u}.$
\end{lemma}

Notice here that $f^{\ast}$ is well defined in $\left(  E^{u}\right)  _{s}$
and Theorem \ref{L4} will be used to prove that $f^{\ast}$ is in fact in
$E^{u}.$

\begin{proof}
Put $M=\sup\limits_{n}T\left\vert f_{n}\right\vert \in M.$ By Doob Maximal
Inequality \cite[Theorem 6.1]{L-06} we get $TP_{\left(  f_{n}-\lambda
e\right)  ^{+}}e\leq\lambda^{-1}M.$ Since $P_{\left(  f_{n}-\lambda e\right)
^{+}}e\uparrow P_{\left(  f^{\ast}-\lambda e\right)  ^{+}}e$ it follows that%
\[
TP_{\left(  f^{\ast}-\lambda e\right)  ^{+}}e\leq\lambda^{-1}M,
\]
and then $T\left(  \inf\limits_{\lambda>0}P_{\left(  f^{\ast}-\lambda
e\right)  ^{+}}e\right)  =0.$ Now using the strict positivity of $T$ to get
$\inf\limits_{\lambda>0}P_{\left(  f^{\ast}-\lambda e\right)  ^{+}}e=0,$ and
so the result follows by applying Theorem \ref{L4}.
\end{proof}

Before stating and proving our last result in this section (Proposition
\ref{L2}), that will be needed to prove the main theorem of this paper, let us
say a few words about Riemann integral in vector lattices. This notion had
been introduced by the authors in \cite{L-311}. The theory is a faithful
generalization of the classical theory. From that paper we invoke the
following result \cite[Lemma 5]{L-311}.

\begin{lemma}
\label{R1}Let $X$ be an order complete vector lattice with a weak order unit
$e.$ Let $p\in\left(  0,\infty\right)  $ and $a,\varepsilon\in\left(
0,\infty\right)  $ with $\varepsilon<a$. Then
\[
x^{p}-\varepsilon^{p}e=\int_{\varepsilon}^{a}pt^{p-1}P_{(x-te)^{+}%
}edt\text{\quad for all }x\in X\text{ with }\varepsilon e<x\leq ae.
\]

\end{lemma}

Now if $p\in\lbrack1,\infty),$ the same proof yields the following:

\begin{lemma}
\label{L9} Let $X$ be an order complete vector lattice with a weak order unit
$e.$ Let $a\in\left(  0,\infty\right)  $ and $p\in\lbrack1,\infty).$ Then%
\[
x^{p}=\int_{0}^{a}pt^{p-1}P_{(x-te)^{+}}edt\text{\quad for all }x\in X\text{
with }0\leq x\leq ae.
\]

\end{lemma}

We recall that for if $E$ is a Dedekind complete vector lattice with a weak
order unit $e$ and if $x\geq e$ then $x$ has an inverse in the $f$-algebra
$E^{u}$ (see \cite[Theorem 3.4]{L-381}), we denote this inverse by $x^{-1}.$
Since $E$ is an ideal in $E^{u}$ and $x^{-1}\leq e$ it turns out that $x^{-1}$
belongs to $E.$

\begin{proposition}
\label{L2}Let $E$ be a Dedekind complete vector lattice with a weak order unit
$e.$ Let $\left(  x_{n}\right)  $ be a sequence in $E^{+}$ and $R_{n}=%
{\textstyle\sum\limits_{i=1}^{n}}
x_{i}.$ Then the series $%
{\textstyle\sum}
x_{n}(e+R_{n})^{-2}$ converges in order and we have%
\[%
{\textstyle\sum\limits_{n=1}^{\infty}}
x_{n}(e+R_{n})^{-2}\leq e.
\]

\end{proposition}

\begin{proof}
We assume first that the sequence $\left(  x_{n}\right)  $ satisfies $0\leq
x_{n}\leq\lambda e$ for some $\lambda\in\left(  0,\infty\right)  $ and for
every $n=1,2....$ By Lemma \ref{L9} we have%
\[
x_{n}=R_{n}-R_{n-1}=\int_{0}^{na}\left(  P_{n,t}-P_{n-1,t}\right)  edt,
\]
where $P_{n,t}=P_{(R_{n}-te)^{+}}.$ Applying Theorem 4 in \cite{L-311} to get%
\[
x_{n}(e+R_{n})^{-2}=\int_{0}^{na}(e+R_{n})^{-2}(P_{n}\left(  t\right)
-P_{n-1}\left(  t\right)  )edt.
\]
Now observe that%
\[
(e+R_{n})^{-2}=\dfrac{e+R_{n})^{-2}}{\left(  1+t^{2}\right)  }\left(
te-R_{n}\right)  \left(  \left(  2+t\right)  e+R_{n}\right)  +\dfrac
{a}{\left(  1+t\right)  ^{2}}\leq
\]
which combined with the fact that $(P_{n,t}-P_{n-1,t})\left(  te-R_{n}\right)
\leq0$ leads to the following:%
\[
(e+R_{n})^{-2}(P_{n,t}-P_{n-1,t})e\leq\left(  1+t\right)  ^{-2}(P_{n,t}%
-P_{n-1,t})e
\]
and then
\[
x_{n}(e+R_{n})^{-2}\leq\int_{\varepsilon}^{na}\frac{1}{(1+t)^{2}}\left(
P_{n}\left(  t\right)  -P_{n-1}\left(  t\right)  \right)  edt.
\]
It follows that
\begin{align*}%
{\textstyle\sum\limits_{k=1}^{n}}
x_{n}(e+R_{n})^{-2}  &  \leq\int_{\varepsilon}^{na}\frac{1}{(1+t)^{2}}%
{\textstyle\sum\limits_{k=1}^{n}}
\left(  P_{k,t}-P_{k-1}\left(  t\right)  \right)  edt\\
&  \leq\int_{\varepsilon}^{na}\frac{dt}{(1+t)^{2}}e\leq e.
\end{align*}
This proves the result for bounded sequences in the ideal $I_{e}$ generated by
$e$. For general case we use our first case to get%
\[%
{\textstyle\sum\limits_{i=1}^{n}}
(x_{i}\wedge\lambda e)(e+R_{i}^{\lambda})^{-2}\leq e,
\]
for all positive real $\lambda,$ where $R_{i}^{\lambda}=R_{n}=%
{\textstyle\sum\limits_{j=1}^{i}}
\left(  x_{j}\wedge\lambda e\right)  .$ It follows that
\[%
{\textstyle\sum\limits_{k=1}^{n}}
x_{n}(e+R_{k})^{-2}\leq e.
\]
As this occurs for each integer $n$ we deduce that
\[%
{\textstyle\sum\limits_{k=1}^{\infty}}
x_{k}(e+R_{k})^{-2}\leq e,
\]
which completes the proof.
\end{proof}

\section{Burkholder Theorem}

We assume throughout that $E$ is a Dedekind complete vector lattice with a
weak order unit $e$ and $T$ is a conditional expectation with $Te=e.$ By
taking $L^{1}\left(  T\right)  $ we will assume also that the space $E$ is
$T$-universally complete.

Consider now a filtration $\left(  T_{n}\right)  _{n\geq1}$ with $T_{1}%
=T$.\textrm{ }If $x=\left(  x_{n}\right)  _{n\geq1}$ is an adapted process we
define the quadratic variation process associated to $x$ by putting%
\[
S_{k}(x)=%
{\textstyle\sum\limits_{j=1}^{k}}
\left(  \Delta x_{j}\right)  ^{2}\quad\text{for all }k\geq1.
\]
Also, we set%
\[
x_{k}^{\ast}=\sup\limits_{1\leq i\leq k}x_{i}\quad\text{for all \ \ }k\geq1.
\]
If $P=\left(  P_{n}\right)  _{n\geq1}$ is a stopping time, the stopped process
$x^{P}$ is defined by%
\[
x_{n}^{P}=%
{\textstyle\sum\limits_{k=1}^{n-1}}
\Delta P_{k}x_{k}+P_{n-1}^{d}x_{n}.
\]
We also put%

\[
x_{n}^{P-1}=%
{\textstyle\sum\limits_{k=1}^{n-1}}
\Delta P_{k}x_{k-1}+P_{n-1}^{d}x_{n-1}.
\]
It follows easily from the definition that%
\begin{equation}
S_{n}^{P}\left(  x\right)  -S_{n}^{P-1}\left(  x\right)  =\left(  \left(
\Delta x_{n}\right)  ^{2}\right)  ^{P}=\left(  \Delta x_{n}^{P}\right)  ^{2},
\label{10}%
\end{equation}
where the last equality follows from the fact that the projections $\Delta
P_{1},...,\Delta P_{n-1},P_{n-1}^{d}$ are disjoint.

The following lemma, taken from \cite{L-311}, is one of the classical
properties of stopped martingales and will be used later.

\begin{lemma}
\label{L1}\textrm{\cite[Lemma 9]{L-311} }Let $(x_{n})_{n\geq1}$ be a
submartingale, and $(P_{k})_{k\geq1}$ be a stopping time then we have%
\[
T\left(  x_{P\wedge n}\right)  \leq T(x_{n})\qquad\text{for all }n=1,2,...
\]

\end{lemma}

We are ready now to state our main result.

\begin{theorem}
\label{Main}Let $E$ be a Dedekind complete Riesz space with weak order unit
$e,$ equipped with a conditional expectation $T$ with $Te=e.$ Let $f=\left(
f_{n}\right)  _{n\geq1}$ and $g=\left(  g_{n}\right)  _{n\geq1}$ be two
martingales in $L^{2}\left(  T\right)  $ with respect to the same filtrartion
$\left(  T_{n}\right)  _{n\geq1},$ with quadratic variation processes $\left(
S_{n}\left(  f\right)  \right)  _{n\geq1}$ and $\left(  S_{n}\left(  g\right)
\right)  _{n\geq1}$ respectively. Assume that $S_{n}(g)\leq S_{n}(f)$ for all
$n=1,2,...,$ and $N_{1}\left(  f\right)  :=\sup\limits_{n}T\left\vert
f_{n}\right\vert \in E^{u}.$ Then $\left(  g_{n}\right)  _{n\geq1}$ is order
convergent in $E^{u}.$
\end{theorem}

The proof of this theorem will be divided in several steps. Before starting
our proof let us introduce some additional notations. Fix a positive real
$\lambda$ and let $P_{n}$ be the band projection the principal band generated
by%
\[
a_{n}=\left(  f_{n}-\lambda e\right)  ^{+}\vee\left(  S_{n}\left(  f\right)
-\lambda e\right)  ^{+}.
\]
As $a_{n}\in R\left(  T_{n}\right)  $ for every $n,$ the sequence $P=\left(
P_{n}\right)  _{n\geq1}$ is a stopping time. Our first step to proving Theorem
\ref{Main} is the following:

\begin{lemma}
\label{L8}Under the hypothesis of Theorem \ref{Main} and using the notations
above, the sequence $(\Delta g_{n}^{P})_{n\geq1}$ is $T$-bounded in $E,$ that
is, there exists $M\in E^{+}$ such that $\sup\limits_{n}T\left\vert \Delta
g_{n}^{P}\right\vert \leq M.$
\end{lemma}

\begin{proof}
Since $P_{k}\geq P_{\left(  S_{k}\left(  f\right)  -\lambda e\right)  ^{+}}$
we have $P_{k}^{d}=P_{k}^{d}P_{\left(  S_{k}\left(  f\right)  -\lambda
e\right)  ^{+}}^{d},$ whence%
\begin{equation}
P_{k}^{d}S_{k}^{2}\left(  f\right)  =P_{k}^{d}P_{\left(  S_{k}-\lambda
e\right)  ^{+}}^{d}S_{k}^{2}\left(  f\right)  \leq P_{k}^{d}\lambda^{2}e
\label{11}%
\end{equation}
This implies that%
\begin{align*}
(S_{n}^{P-1}(f))^{2}  &  =%
{\textstyle\sum\limits_{k=1}^{n-1}}
P_{i}P_{i-1}^{d}S_{i-1}^{2}(f)+P_{n-1}^{d}S_{n-1}^{2}(f)\\
&  \leq%
{\textstyle\sum\limits_{k=1}^{n-1}}
P_{i}P_{i-1}^{d}\lambda^{2}e+\lambda^{2}P_{n-1}^{d}e,
\end{align*}
and so we get
\begin{equation}
(S_{n}^{P-1}(f))^{2}\leq\lambda^{2}e. \label{12}%
\end{equation}
In a similar manner one can prove%
\[
\left\vert f_{n}^{P-1}\right\vert \leq\lambda e.
\]
Moreover, we have%
\begin{equation}
(\Delta g_{n}^{P})^{2}\leq\left(  S_{n}^{P}g\right)  ^{2}\leq\left(  S_{n}%
^{P}f\right)  ^{2}=P_{n}^{d}(S_{n}(f))^{2}+P_{n}(S_{n}^{P}(f))^{2}. \label{13}%
\end{equation}
Now by (\ref{11}), $P_{n}^{d}(S_{n}(f))^{2}\leq$ $\lambda^{2}P_{n}^{d}e$ and
by combining (\ref{10}) and (\ref{12}) we obtain%
\begin{align*}
P_{n}(S_{n}^{P}(f))^{2}  &  =P_{n}(S_{n}^{P-1}(f)^{2})+P_{n}\left(  \Delta
f_{n}^{P}\right)  ^{2}\\
&  \leq\lambda^{2}P_{n}e+P_{n}\left(  \Delta f_{n}^{P}\right)  ^{2}.
\end{align*}
Using these inequalities, (\ref{13}) yields
\[
(\Delta g_{n}^{P})^{2}\leq\lambda^{2}e+P_{n}\left(  \Delta f_{n}^{P}\right)
^{2}\leq\left(  P_{n}\left\vert \Delta f_{n}^{P}\right\vert +\lambda e\right)
^{2}.
\]
\newline Thus,%
\begin{align*}
\left\vert \Delta g_{n}^{P}\right\vert  &  \leq P_{n}\left\vert \Delta
f_{n}^{P}\right\vert +\lambda e\leq P_{n}\left\vert f_{n}^{P}\right\vert
+P_{n}\left\vert f_{n}^{P-1}\right\vert +\lambda e\\
&  \leq P_{n}\left\vert f_{n}^{P}\right\vert +2\lambda e\leq\sup
\sum\limits_{i=1}^{n}\Delta P_{i}\left\vert f_{i}\right\vert +2\lambda e.
\end{align*}
Observe that $\sup\sum\limits_{i=1}^{n}\Delta P_{i}\left\vert f_{i}\right\vert
$ exists in $E^{u}$ as the sequence $\left(  \Delta P_{i}\left\vert
f_{i}\right\vert \right)  _{i\in\mathbb{N}}$ is disjoint. We need only show
that this supremum belongs to $E,$ or, in other words, that the increasing
sequence $\left(  \sum\limits_{i=1}^{n}\Delta P_{i}\left\vert f_{i}\right\vert
\right)  _{n\geq1}$ converges in order in $E.$ Since the space $E$ is
$T$-universally complete it will be enough to prove that $\sup\limits_{n}%
T\sum\limits_{i=1}^{n}\Delta P_{i}\left\vert f_{i}\right\vert $ exists in
$E^{u}$. To this end applying Lemma \ref{L1} to the submartingale $\left\vert
f_{n}\right\vert $ to obtain,%
\[
\sup\limits_{n}T\sum\limits_{i=1}^{n}\Delta P_{i}\left\vert f_{i}\right\vert
\leq\sup\limits_{n}T\left\vert f_{_{n}}^{P}\right\vert \leq\sup\limits_{n}%
T\left\vert f_{_{n}}\right\vert \leq N_{1}(f).
\]
which proves that the sequence $\left(  \sup\Delta g_{n}^{P}\right)  _{n\geq
1}$ is $T$-bounded in $E$ as required$.$ Moreover we have
\[
T\left\vert \sup\Delta g_{n}^{P}\right\vert \leq2\lambda e+N_{1}\left(
f\right)  .
\]
and the proof is finished.
\end{proof}

Define $U_{0}=e$ and for $n\geq1,$%
\[
U_{n}=e+S_{n\wedge P}^{2}(g),
\]
We know by \cite[Theorem 3.4]{L-381} that $U_{n}$ is invertible in the
$f$-algebra $E^{u}$ and that $U_{n}^{-1}\in E$, as $E$ is an ideal in $E^{u}$
and $U_{n}^{-1}\leq e.$

\begin{lemma}
\label{L3}With the notation as above, the series $%
{\textstyle\sum}
U_{n}^{-1}\Delta g_{n}^{P}$ is order convergent in $E$.
\end{lemma}

\begin{proof}
Put $A_{n}=U_{n}^{-1}\Delta g_{n}^{P}$ and $h_{n}=%
{\textstyle\sum\limits_{i=1}^{n}}
\left(  I-T_{i-1}\right)  A_{i}$ and observe that $h_{n}$ belongs to
$R(T_{n})$ and that $T_{n-1}\Delta h_{n}=T_{n-1}\left(  A_{n}-T_{n-1}%
A_{n}\right)  =0.$ This shows that $\left(  h_{n}\right)  _{n\geq1}$ is a
martingale. Since $T_{n-1}$ is an averaging operator we have%
\begin{align*}
T_{n-1}(\Delta h_{n})^{2}  &  =T_{n-1}(A_{n})^{2}+(T_{n-1}A_{n})^{2}%
-2A_{n}T_{n-1}(A_{n})\\
&  =T_{n-1}\left(  A_{n}{}^{2}\right)  -T_{n-1}(A_{n}.T_{n-1}A_{n})\leq
T_{n-1}\left(  A_{n}{}^{2}\right)  .
\end{align*}
It follows that%
\begin{align*}%
{\textstyle\sum\limits_{i=1}^{n}}
T(\Delta h_{n})^{2}  &  =%
{\textstyle\sum\limits_{i=1}^{n}}
TT_{i-1}(\Delta h_{n})^{2}\leq%
{\textstyle\sum\limits_{i=1}^{n}}
TT_{i-1}\left(  A_{n}{}^{2}\right) \\
&  =%
{\textstyle\sum\limits_{i=1}^{n}}
T(A_{n})^{2}.
\end{align*}
Apply Lemma \ref{L2} to the sequence $x_{n}=\left(  \Delta g_{n}^{P}\right)
^{2}$ to prove that $\left(  h_{n}\right)  _{n\geq1}$ is bounded in
$L^{2}\left(  T\right)  $ and then in $E$ (\textrm{\cite[Theorem 3.2]{L-180}%
}).\textrm{ }Using \cite[Theorem 3.5]{L-03} we derive that $\left(
h_{n}\right)  _{n\geq1}$ is order convergent. Now, to prove the lemma it will
be sufficient to show that the series $%
{\textstyle\sum}
T_{i-1}A_{i}$ is order convergent. Since%
\[
T_{i-1}U_{i-1}^{-1}\Delta g_{i}^{P}=U_{i-1}^{-1}T_{i-1}\Delta g_{i}^{P}=0,
\]
we have
\begin{align*}
T%
{\textstyle\sum\limits_{i=1}^{n}}
\left\vert T_{i-1}A_{i}\right\vert  &  =T%
{\textstyle\sum\limits^{n}}
\left\vert T_{i-1}\left(  U_{i}^{-1}-U_{i-1}^{-1}\right)  \Delta g_{i}%
^{P}\right\vert \\
&  \leq T%
{\textstyle\sum\limits_{i=1}^{n}}
T_{i-1}\sup_{k}\left\vert \Delta g_{k}^{P}\right\vert \left(  U_{i-1}%
^{-1}-U_{i}^{-1}\right) \\
&  =T\sup_{k}\left\vert \Delta g_{k}^{P}\right\vert \left(  e-U_{n}%
^{-1}\right)  \leq T\sup_{k}\left\vert \Delta g_{k}^{P}\right\vert .
\end{align*}
By Lemma \ref{L8} the sequence $T%
{\textstyle\sum\limits_{i=1}^{n}}
\left\vert T_{i-1}A_{i}\right\vert $ is bounded in $E^{u}.$ As $E$ is
$T$-universally complete we deduce that the sequence $%
{\textstyle\sum\limits_{i=1}^{n}}
\left\vert T_{i-1}A_{i}\right\vert $ is order convergent in $E,$ which proves
the lemma.\medskip
\end{proof}

\begin{proof}
[Proof of the main Theorem]As $f$ is a strongly bounded martingale in $E^{u}$.
Austin lemma \cite[Section 4]{L-21} yields that $S:=\sup\limits_{n}%
S_{n}\left(  f\right)  $ exists in $E^{u}$. We have already shown that the
series $%
{\textstyle\sum}
\left\vert U_{n}^{-1}\Delta g_{n}^{P}\right\vert $ is order convergent. Now
since the sequence $\left(  U_{n}^{-1}\right)  _{n\geq1}$ is increasing and
bounded it follows that the series $%
{\textstyle\sum}
\Delta g_{n}^{P}$ is order convergent, which implies the convergence of the
sequence $\left(  g_{n}^{P}\right)  _{n\geq1}$. Our purpose now is to prove
the convergence of the sequence $\left(  g_{n}\right)  $. Notice that we have
worked up to now with the fixed stopping time $P=\left(  P_{(\left(
\left\vert f_{n}^{\ast}\right\vert -\lambda e\right)  ^{+}\vee\left(
S_{n}\left(  f\right)  -\lambda e\right)  ^{+}}\right)  _{n\in\mathbb{N}}$ for
fixed $\lambda>0.$ Next, we will consider a sequence of stopping times and it
would be suitable to introduce some additional notations:%
\[
P_{n,k}=P_{(\sup\limits_{i\leq n}\left(  \left\vert f_{i}\right\vert
-ke\right)  ^{+}\vee\left(  S_{n}\left(  f\right)  -ke\right)  ^{+}}%
\qquad\text{ and}\qquad P\left(  k\right)  =\left(  P_{n,k}\right)  _{n\geq
1}.
\]
Theorem \ref{ThK} will be applied to the sequence $\left(  Q_{k}^{d}\right)  $
where $Q_{k}$ is the band projections defined by%
\[
Q_{k}=P_{(f^{\ast}-ke)^{+}\vee\left(  S\left(  f\right)  -ke\right)  ^{+}%
}=\sup\limits_{n}P_{n,k}.
\]
Thanks to Lemma \ref{L5} and Austin's Theorem for vector lattices
\cite[Theorem 4.1]{L-21} we deduce that $S\left(  f\right)  $ and $f^{\ast
}=\sup\left\vert f_{i}\right\vert $ belongs to $E^{u}.$ It is easily seen that
$Q_{k}^{d}g_{n}^{P\left(  k\right)  }=Q_{k}^{d}g_{n}$ for all integers $k$ and
$n.$ Thus, the sequence $\left(  Q_{k}^{d}g_{n}\right)  _{n\geq1}$ is order
convergent for every $k.$ Moreover, it follows from Theorem \ref{L4} that
$Q_{k}^{d}\uparrow I.$ The proof is now completed by invoking Theorem
\ref{ThK}.
\end{proof}

\end{document}